\documentclass{amsart}
\usepackage{bbm}

\usepackage{enumerate}
\usepackage[hyperindex=true,plainpages=false,colorlinks=false,pdfpagelabels]{hyperref}
\usepackage{verbatim}

\usepackage[all]{xy}

\theoremstyle{definition}

\theoremstyle{remark}

\theoremstyle{plain}
\newtheorem{The}{Theorem}
\newtheorem*{The*}{Theorem}
\newtheorem{Pro}{Proposition}
\newtheorem{Lem}{Lemma}
\newtheorem{Cor}{Corollary}
\newtheorem*{Cor*}{Corollary}

\theoremstyle{definition}
\newtheorem*{Def}{Definition}
\newtheorem{Rem}{Remark}
\newtheorem{Exa}{Example}
\newtheorem*{Rem*}{Remark}

\numberwithin{equation}{section}

\DeclareMathOperator{\del}{\partial}

\newcommand{\dvector}[1]{{\left(\begin{matrix}#1\end{matrix}\right)}}

\newcommand{\R}{\mathbb{R}}

\newcommand{\C}{\mathbb{C}}
\newcommand{\N}{\mathbb{N}}
\newcommand{\Z}{\mathbb{Z}}

\renewcommand{\H}{\mathbb{H}}  

\newcommand{\CP}{\mathbb{CP}}

\newcommand{\ii}{\mathbbm i}
\newcommand{\jj}{\mathbbm j}
\newcommand{\kk}{\mathbbm k}

\renewcommand{\Im}{\operatorname{Im}}
\renewcommand{\Re}{\operatorname{Re}}
\newcommand{\II}{\operatorname{II}}

\begin{document}

\title{Equivariant constrained Willmore tori in the $3-$sphere}

\author{Lynn Heller}
\address{Institut f\"ur Mathematik,  Universit{\"a}t T\"ubingen, Auf der Morgenstelle 10, 72076 T\"ubingen, Germany\\}
\email{lynn-jing.heller@uni-tuebingen.de}       

\date{\today}

\thanks{The author is supported by the Sonderforschungsbereich Transregio 71}

\begin{abstract} 
In this paper we study equivariant  constrained Willmore tori in the $3-$sphere. These tori admit a $1-$parameter group of M\"obius symmetries and are critical points of the Willmore energy under conformal variations. 
We show that the spectral curve  associated  to an equivariant torus is given by a double covering of $\C$ and classify equivariant constrained Willmore tori by the genus $g$ of their spectral curve.  In this case the spectral genus satisfies $g \leq 3.$
 \end{abstract}
\maketitle


\section{Introduction}\label{sec:intro}

We consider conformally immersed surfaces into $3-$dimensional space forms which are critical points of the Willmore functional $\int (H^2 +1)dA$  under conformal variations, see \cite{BoPetP} and \cite{KS}. These surfaces are called constrained Willmore. They are natural generalizations of Willmore surfaces and compact constrained Willmore minimizers can be viewed as optimal 
realizations of the underlying Riemann surface in $3-$space. While minimal surfaces are always Willmore,  
surfaces with constant mean curvature (CMC surfaces) provide examples of constrained Willmore surfaces. If the underlying Riemann surface is a topological sphere, there is only one conformal class and all constrained Willmore surfaces are Willmore.
In the case of tori Bohle \cite{B} has shown that constrained Willmore surfaces form an integrable system. This means that we can associate to every constrained Willmore torus a compact Riemann surface of genus $g$ -  the spectral curve -  and the immersion is then explicitly given in terms of algebraic data on the spectral curve. The number $g$ is called the spectral genus of the immersion.\\

The first examples of Willmore tori which are not minimal in a space form have been the Willmore Hopf tori by Pinkall \cite{P}. These tori are given by the preimages of elastic curves on $S^2,$ critical points of the energy functional $\int\kappa^2 ds$, under the Hopf fibration. The notion of equivariant tori generalizes this construction. Equivariant tori are conformally immersed tori into the conformal $3-$sphere with a $1-$parameter family of M\"obius symmetries.  Another class of examples of equivariant surfaces are tori of revolution given by rotating closed curves in the upper half plane around the $x-$axis. The torus is then constrained Willmore if the curve is elastic in the upper half plane, viewed as the hyperbolic plane. Constrained Willmore tori of revolution are always CMC  in a space form, see \cite{B}. Elastic curves in $S^2$ and $H^2$ are constructed in \cite{LS}. Equivariant Willmore tori were classified  in \cite{FP}. The equivariant case is especially interesting since the tori are comparably easy to obtain and we get a large class of explicit examples. These examples can be used to investigate Whitham deformations of constrained Willmore tori. In analogy to equivariant harmonic tori into the $2-$sphere, i.e., the CMC case,  it has been conjectured that all equivariant constrained Willmore tori have spectral genus $1,$ since the known examples are given in terms of elliptic functions. We disprove this conjecture but we show that equivariant constrained Willmore tori have spectral genus $g\leq3.$\\

The paper is organized as follows. In the second section we deal with general properties of equivariant and constrained Willmore tori. We show that every equivariant torus can be interpreted as the preimage of a curve  under a certain Seifert fibration of a conformally flat $S^3$. The geometry of the surface is determined by the geometry of the curve and of the fibration.  We state the Euler-Lagrange equation for the constrained Willmore problem in the equivariant case and define an associated family of constrained Willmore surfaces. \\

We recall the notion of the spectral curve $\Sigma$ of a general conformally immersed torus $f: T^2 \rightarrow S^3$ of \cite{BLPP}  in the third section. We call $f$ a finite gap immersion, if $\Sigma$ has finite genus.
To the spectral curve one can associate a $T^2-$family of holomorphic line bundles $\mathcal L_x \rightarrow \Sigma.$  It is shown in theorem 5.6 of \cite{BoPP}  that for a fixed $x_0 \in T^2$ the map 
$$\Psi: T^2 \rightarrow \text{Jac}(\Sigma) \quad  x   \mapsto \mathcal L_x \mathcal L_{x_{0}}^{-1}$$
is a group homomorphism. 
The immersion can be reconstructed from the spectral curve and $\Psi(T^2).$ \\

In the last section we prove that the spectral curve of an equivariant immersion is a double covering of $\C$.
Isothermic equivariant tori in $S^3$ are constrained Willmore if and only if they  have  spectral genus $g \leq1$.  
Moreover,  they
have constant mean curvature in some $3-$dimensional space form and are associated to a constrained Willmore cylinder of revolution as isothermic constrained Willmore surfaces.
A non-isothermic equivariant torus in $S^3$ is constrained Willmore if and only if its spectral genus satisfies $g=2$ or 
$g=3$. In particular,  an equivariant torus has spectral genus $2$ if and only if it is associated to a non- homogenous constrained Willmore Hopf cylinder, which is given by the preimage of a not necessarily closed curve in $S^2$ under the Hopf fibration,  as constrained Willmore surfaces.
  
\section{Equivariant Constrained Willmore Tori in the $3-$Sphere}
We consider the $S^3 \subset \H$  as the set of unit length quaternions. The isometries of $S^3$ are then of the form
$$a \mapsto \lambda a \mu,$$
for $ \lambda , \mu \in S^3$.   Further we identify $\H = \C \oplus \C \jj.$  
\begin{Def}
A map $f: \C \rightarrow S^3 $  is called $\R-$equivariant, if there exist group homomorphisms 
\begin{equation*}
\begin{split}
&M: \R \rightarrow \text{M\"ob}(S^3), t \mapsto M_t,\\
&\tilde M: \R \rightarrow \{\text{conformal transformations of } \C\}, t \mapsto \tilde M_t,
\end{split}
\end{equation*}
such that 
$$f \circ \tilde M_t = M_t \circ f,  \text{ for all } t.$$
Here M\"ob$(S^3)$ is the group of M\"obius transformations of $S^3.$
\end{Def}

\noindent
If $f$ is doubly periodic with respect to a lattice $\Gamma \subset \C,$ we obtain a torus and the following proposition holds.

\begin{Pro}\label{equivariant}
Let $f: T^2 \cong \C/\Gamma \rightarrow S^3$ be a equivariant conformal immersion. Then there exist a holomorphic coordinate $z = x+iy$ of $T^2$ together with $m , n \in \N$ and $gcd(m,n) = 1$ such that 
$$f(x,y) = e^{  \ii \tfrac{m+n}{2}x} f(0,y) e^{\ii  \tfrac{m-n}{2} x}.$$
up to isometries of $S^3$ and the identification of $S^3$ with $SU(2).$
\end{Pro}
\begin{Rem}
Conjugation in the proposition above is equivalent to a M\"obius transformation of the ambient space.
\end{Rem}
\begin{proof}
A proof of the proposition can be found in the sections 5.1 and 5.2 of \cite{H}.
\end{proof}
\begin{Def}
Let $l_1:=  \tfrac{m+n}{2}$ and $l_2 :=  \tfrac{m-n}{2}$. A immersion $f: T^2 \rightarrow S^3$ of the form 
$$f(x,y) = e^{\ii l_1 x} f(0,y) e^{\ii l_2 x},$$
is called a $(m,n)-$torus 
and $\gamma := f(0,y)$ its profile curve, which is in general not closed.
\end{Def}

A conformally parametrized surface in $S^3$ has two conformal invariants which determine the surface up to M\"obius transformations, see  \cite{BuPP}.  The first one is the conformal Hopf differential $q$. The second is the Schwarzian derivative $c$. In the equivariant case $c$ is determined by $q$ up to a integration constant by the Gau\ss-Codazzi equations. Thus we will only consider $q$ in the following. In contrast to \cite{BuPP} we consider the conformal Hopf differential as a complex valued function by trivializing $K$ via $dz$.
\begin{Def}
Let $f$ be a conformally immersed surface in $S^3.$ We call the function
 $$q := \frac{\II \left(\frac{\del }{\del z},\frac{\del}{\del z} \right)}{|df|}$$
 the conformal Hopf differential of $f.$  
\end{Def}

\noindent
The map $f$ is isothermic, i.e., $f$ has a conformal curvature line parametrization, if there is a $\mu \in S^1$ such that $q\mu$ is real valued. 
\subsection{Seifert Fibrations of $S^3$}\label{Seifert fibrations}
Next we want to interpret the conformal Hopf differential of an equivariant immersion in terms of invariants of its profile curve and its $(m,n)-$type. 
For this we introduce the Seifert fiber spaces. For $m,n\in \N$ coprime, consider the following equivalence relation on $S^3.$
Let $a, b \in S^3,$ then
 $$ a \sim_{m,n} b \Leftrightarrow \text{ if  there exist a } t \text{ such that } a= e^{\ii l_1 t} b e^{\ii l_2 t}.$$ 

\begin{Def}
The triple $ P:= (S^3, S^3/_{\sim_{m,n}} , \pi_{m,n}),$ where $\pi_{m,n}$ maps every point in $S^3$ to its equivalence class is a $(m,n)$-Seifert fiber space.  
\end{Def}

\noindent
The base space $S^3/_{\sim_{m,n}}$  is a regular manifold away from the points $[1]_{\sim_{m,n}}$ and $[\jj]_{\sim_{m,n}}.$ If $mn > 1$ both points are singular. For $m= n = 1$ the projection $\pi_{1,1}$ is the Hopf fibration. Thus the base space is the round sphere and has no singular points. In the case of $m=0, n= 1$ we have that only  $[\jj]_{\sim_{m,n}}$ is singular, see \cite{BW} for details. By construction the following Proposition holds.
\begin{Pro}
There exist an one-to-one correspondence between closed curves in the base space of the $(m,n)-$Seifert fiber space and $(m,n)-$tori.
\end{Pro}
\begin{Exa}\label{DH}
The two exceptional cases here  provide the easiest examples of equivariant tori.
In the first case $m = n = 1$  the torus is obtained as the  preimage of a closed curve on $S^2$ under the Hopf fibration. These so called Hopf tori were first studied in \cite{P}. In the second case we have $m = 1, n= 0$ and  $M_t$ is a rotation. The torus can be constructed by the the rotation of a closed curve in the open upper half plane  around the $x-$ axis. 
\end{Exa}

$ P$ is a principal fiber bundle away from the singular points. Let $P^* := P\setminus\{\text{singular points} \}.$ We define a connection on $P^*$ such that the curvature of $P^*$ has a geometric meaning for the corresponding $(m,n)-$tori. In order to do so, we need  to define a metric $g_{m,n}$. It turns out that the right choice of the metric is given by dividing point wise  the round metric $g$ on $S^3$ by the fiber length $h(a) := |\ii l_1 a|  + l_2a \ii|^2$, i.e., $g_{m,n} = \tfrac{1}{h}g$. \\

Consider the closed curve $\tilde \gamma$ in $S^3/\sim_{m,n}.$ Then a conformal parametrization of the corresponding equivariant torus can be obtained by the horizontal lift of the curve with respect to the principal $S^1-$connection given by its connection $1-$form $\omega_{m,n} = g_{m,n}(., B),$ where $B$ is the fiber direction. In other words, we  take a lift of the curve $\tilde \gamma$  to $S^3$ such that its tangent vector is point wise perpendicular to the fiber of $P.$ This defines also a horizontal lift of the frame of $\tilde \gamma$. The curvature of the lifted curve $\gamma$ in  the direction given by the horizontal lift  of the normal vector of $\tilde \gamma$ is by definition the curvature of the curve $\tilde \gamma$.  Further  the lifted curve  $\gamma$ has torsion in the fiber direction $B,$ which becomes the binormal vector of $\gamma$. 
Let $T$ denote the tangent vector of $\gamma$. Then we have $$<B', T > = <l_1\ii \gamma' + l_2 \gamma' \ii, \gamma'> = \text{Re}((l_1\ii \gamma' + l_2 \gamma' \ii)\bar \gamma') = 0,$$ where $()'$ is the derivative in $\H \cong \R^4.$ Thus we have
\begin{Lem}\label{frenet}
The frame of the profile curve $\gamma$ given by the tangent vector  $T$, the normal vector $N_{orm}$ and the fiber direction $B$ of the Seifert fiber space is its Fr\'enet frame in $S^3$. 
\end{Lem}

\noindent
 A straight forward computation shows the following.

\begin{Pro}\label{confhopfdiff}
Let $f$ be a $(m,n)-$torus with normal vector $N_{orm}$ and 
conformal Hopf differential $q.$ Then 
$$q  = \frac{<f_{xx} - f_{yy} - 2 \ii f_{xy}, N_{orm} >}{\sqrt h} = \frac{1}{4}(\kappa_{m,n} +  i \Omega),$$
where $\kappa_{m,n}$ is the curvature of the profile curve of $f$ with respect to $g_{m,n}$ and   \linebreak  $\Omega vol_{m,n} = \frac{2mn}{\sqrt h} vol_{m,n} $ is the curvature form of the principal $S^1-$connection given by the connection $1-$form $\omega = g_{m,n}(., B)$.
\end{Pro} 
\begin{Rem}\label{kurvenq}
In particular,  the conformal Hopf differential $q$  of equivariant tori  depends only on the curve parameter and is periodic.
\end{Rem}

\begin{Exa}
In the case of tori of revolution, i.e., $m = 1$ and $n= 0,$ we have that $4q = \kappa_{1,0} = \kappa$ is real and $\kappa$ is the curvature of the profile curve $\gamma$ in the hyperbolic plane. These tori are always isothermic. In the case of Hopf tori, i.e., $m = n = 1,$ we have that  the base space is the round $2-$sphere of constant curvature $4$ and the curvature of the Seifert fiber space is $2$. Thus $4q = \kappa  + 2i,$ where $\kappa$ is the curvature of $\tilde \gamma = \pi_{1,1}(\gamma)$ in this $2-$sphere.
\end{Exa}

\subsection{The Euler-Lagrange Equation of Equivariant Constrained Willmore Tori}

\begin{Def} Let $M$ be a compact surface and let $f: M \rightarrow S^3$ be an immersion into the round sphere. The \emph{Willmore energy} of $f$ is defined to be
$$\mathcal W (f) = \int_ M( H^2 + 1 ) dA,$$
where $H$ is the mean curvature of $f$ and $dA $ is  induced volume form. \\

A conformal immersion $f: M \rightarrow (S^3, g)$ is called Willmore, if it is a critical point of the Willmore energy $W$ under all variations by immersions and it is called constrained Willmore, if it is a critical point of $W$ under conformal variations, see \cite{BoPetP} and \cite{KS}.
\end{Def}

\noindent
In the case of a $(m,n)-$torus the Willmore energy reduces to:
$$ \mathcal W (f)  =  16 \pi \int_0^l |q|^2 dy = \tfrac{1}{2}\pi \int_0 ^l (\kappa_{m,n}^2 + \Omega^2) dy,$$
where $y$ the arclength parameter of the profile curve with respect to $g_{m,n}$.

\begin{Rem}
Instead of considering equivariant constrained Willmore tori it is equivalent to consider curves in $S^3/\sim_{m,n}$ which are critical under the generalized energy functional $\int_0 ^l (\kappa_{m,n}^2 + \Omega^2) dy.$
\end{Rem}

\begin{The}[\cite{BuPP}]
Let $f: T^2\cong \C/\Gamma \rightarrow S^3$ be a conformally parametrized equivariant immersion and $q$ its conformal Hopf differential. Then $f$ is constrained Willmore if and only if there exists a $\lambda \in \C$ such that $q$ satisfies the equation: 
\begin{equation}\label{EL}
\begin{split}
q'' +  8 (|q|^2 + &C)q - 8 \xi q = 2 Re(\lambda q),\\
2\xi' &= \bar q'q - q' \bar q,
\end{split}
\end{equation}
\noindent
where $\xi$  is a purely imaginary function and $C$ a real constant and the derivative is taken with respect to  the profile curve parameter.
\end{The}
The real part of equation (\ref{EL}) is the actual Euler-Lagrange equation. The imaginary part of the equation is the Codazzi equation and the equation on $\xi$ is the Gau\ss  $\ $equation. The function $\xi$ we use is given by $\xi = i \tfrac{mn}{4}H,$ where $H$ is the mean curvature of the immersion in $S^3$.

\begin{Exa}\label{elastic}
In the case of tori  of revolution, we have $4q = \kappa,$ where $\kappa$ is the curvature of $\gamma$ in the hyperbolic plane modeled by the upper half plane.  The Euler-Lagrange equation reduces to the equation
$$\kappa'' + \tfrac{1}{2}\kappa^3 -\kappa = \lambda_1 \kappa,$$
with $\lambda_1\in \R$. This is the Euler-Lagrange equation for elastic curves in the hyperbolic plane, i.e., for critical points of the energy functional $E(\gamma) = \int_\gamma \kappa^2 ds $ with prescribed length. Free elastic curves corresponds to Willmore tori. \\ 

For Hopf tori we have that $4q = \kappa + 2i,$ where $\kappa$ is the curvature of the curve $\pi_{1,1}(f)$ in the round $S^2$ with curvature $4.$ The Euler-Lagrange equation for constrained Willmore tori reduces to
$$\kappa'' + \tfrac{1}{2}\kappa^3 + 2\kappa= \lambda_1 \kappa + \lambda_2,$$
with $\lambda_1, \lambda_2 \in \R.$
This is the Euler-Lagrange equation for constrained elastic curves in the round $S^2$ with curvature $4,$ where $(\lambda_1 + 2)$ is the length and $\lambda_2$ is the enclosed area constraint. Note that free elastic curves do not correspond to Willmore Hopf tori. For free elastic curves we have $\lambda_1 = -2$ and $\lambda_2 = 0,$ but for Willmore Hopf tori $\lambda_1 = \lambda_2 = 0.$
\end{Exa}

 \subsection[Associated Family]{Associated Family and Equivariant Constrained Willmore Immersions}\label{associated}
Let $f: T^2 \rightarrow S^3$ be a conformally immersed  constrained Willmore torus with conformal Hopf differential $q$.   Consider $f$ as a immersion from $\C$ into $S^3$ which is doubly periodic. By relaxing both periodicity conditions we obtain for $\mu \in S^1$ a circle worth of associated constrained Willmore surfaces $f_\mu$ to $f,$ the so called constrained Willmore associated family as follows.\\

\noindent
Let $q$ be a solution of \eqref{EL} and 
let $q_\mu,$ $\mu \in S^1$ be a family of complex functions given by
$$q_\mu =  q \mu$$  
Moreover, let
\begin{equation*}
\begin{split}
C_\mu &=  C + \text{Re}(\mu^2-1 \bar \lambda)\\
\xi_\mu &=  \xi + \text{Im}(\mu^2-1 \bar \lambda)\\
\lambda_\mu & = \bar\mu^2\lambda.
\end{split}
\end{equation*}

\noindent
Then $q_\mu$ satisfies equation \eqref{EL}  with parameters $C_\mu, \lambda_\mu$ and function $\xi_\mu$. In particular, the function $q_\mu$ and $\xi_\mu$ satisfies the Gau\ss-Codazzi equations for surfaces in $S^3$. Thus there exist a family of surfaces $f_\mu$ with conformal Hopf differential $q_\mu$ and mean curvature given by $\xi.$ The so constructed surfaces are automatically constrained Willmore.\\ 

\noindent
Further, surfaces with the same conformal Hopf differential and the same Schwarzian derivative (which is determined by the function $\xi$) differs only by a M\"obius transformation. Since both invariants depends only on one parameter if the initial surface is equivariant, 
the  surfaces in the associated family of equivariant  constrained Willmore surfaces are also equivariant and constrained Willmore. In general these surfaces are not compact, i.e., $f_\mu: \C \rightarrow S^3$ is not doubly periodic, even if the initial surface is.\\

A special class of constrained Willmore tori are the CMC tori. The following theorem characterizes all equivariant CMC tori.  
\begin{The}[\cite{R}]\label{CMC}
An equivariant constrained Willmore torus $f$ is isothermic if and only if  f is an equivariant CMC torus in a space form. In particular if $mn \neq 0,$ then $f$ is CMC in $S^3.$
 \end{The}
\begin{proof}
We only show the second part of the assertion. A surface is isothermic if and only if there exists a $\mu \in S^1$ with $ q \mu$ real valued. Thus $\xi' \equiv 0,$ see equation (\ref{EL}) Since $4\xi = imn H,$ we obtain that these surfaces are CMC in $S^3$ for $mn \neq 0$.
\end{proof}

\begin{The}
Except for  $q = const$, i.e.,  the corresponding torus is homogenous,  Hopf tori are never isothermic.
\end{The}
\begin{proof}
Since for $m= n= 1$ the curvature $\Omega$ of the Seifert fiber space is constant so the imaginary part of $q$ is constant. Thus Im$(q)$ is a constant multiple of  Re$(q)$ if and only if Re$(q)$ is constant. But Re$(q)$ is the curvature of the profile curve on the round $S^2,$ which is constant if and only if the curve is an arc of a circle. 
\end{proof}

For CMC tori there exists a further associated family - the CMC associated family. By definition isothermic surfaces have a conformal curvature line parametrization. In this parametrization the second fundamental form is diagonal and thus the Hopf differential is real valued and $\xi = 0$ by equation (\ref{EL}).
The Euler-Lagrange equation reduces to
$$q'' + 8q^3 + Cq = H q,$$
where $H$ is the constant mean curvature. The associated family for CMC tori is given by $C_r = C + r$ and $H_r = H+r$ for $r \in \R.$ This choice of parameters changes the Schwarzian derivative of the surface but not the conformal Hopf differential and satisfies the Gau\ss-Codazzi equations.  
Putting both associated families together we obtain

\begin{The}
Every isothermic and equivariant torus is in the associated family of a constrained Willmore cylinder of revolution. 
\end{The}
\begin{proof}
First rotate $q$ by $\mu$ such that $q \mu$ is real and  then
choose an $r \in \R$ such that $C = -\tfrac{1}{4}.$ The corresponding $q$ is the conformal Hopf differential of a  cylinder of revolution.  By Theorem (3.3) of \cite{BuPP} we have that isothermic surfaces with the same conformal Hopf differential lie in the same associated family as isothermic surfaces. This associated family coincides in our case with the associated family of CMC surfaces.
\end{proof}

\section{Spectral Curves for Conformal Immersions into $S^4$ with Trivial Normal Bundle}
We use the model of the $4-$sphere given by the projective geometry of the quaternionic projective line $\H P^1,$  
see \cite{BuFLPP} for further reading.
The map $f$ is given by a quaternionic line subbundle $L := f^*\mathcal T$ of the trivial $\H^2-$bundle $V := T^2 \times \H^2, $ where $\mathcal T$ is the tautological bundle of $\H P^1$. Using affine coordinates we can consider $S^3 \subset \H \hookrightarrow \H P^1.$ 
An oriented round $2-$sphere in $S^4$ is given by a linear map $\mathbbm S: \H^2 \rightarrow \H^2$ with $\mathbbm S^2 = -1,$ i.e., a complex structure on $\H^2$. For given $f$  there exist a natural map $S$ from the torus into the space of oriented $2-$spheres - the conformal Gau\ss $ $  map. It assigns to every point $p$ of the surface a $2-$sphere $\mathbbm S_p$ through the point $f(p)$ with the same tangent plane and the same mean curvature.  The pair $(V, \mathbbm S)$ is then a complex quaternionic vector bundle.
 The property $f(p) \in \mathbbm S_p$ translated into the quaternionic language is   $\mathbbm S_p L_p \subset L_p.$ 
Thus the projection of $S$ to the quotient bundle $V/L$ defines a complex structure on the quaternionic line bundle $V/L$ . Moreover, the projection of the $\bar K-$part of the trivial connection on $(V,\mathbbm S)$  to $V/L$ gives a 
quaternionic linear map 
$$D: \Gamma(V/L) \rightarrow \Gamma(\bar K V/L)$$
satisfying the Leibniz rule. Such a map is also called a quaternionic holomorphic structure and the sections in the kernel of $D$ are called holomorphic sections. Let $(1,0)$ and $(0,1)$ be the canonical basis of $\H^2$ viewed as constant sections of $V.$ Then the projection of both sections to $V/L$ are holomorphic. Further the quotient of these sections is $f$ since
$$\pi_{V/L}(0,1) + \pi_{V/L}(1,0)f = 0_{V/L}.$$

\begin{Def}
A section $\psi$ with monodromy of $V/L$ is a section of  the pull-back  $\widetilde{V/L}$ of $V/L$ to the universal covering $\C$ of $T^2 = \C/\Gamma$ with 
$$\omega^*\psi(z) = \psi (z+ \omega) = \psi(z) h_\omega, \text{ for all } z \in T^2 \text { and } \omega \in \Gamma ,$$
where $h: \Gamma \rightarrow \H_*$ is a representation. We call $h$ the monodromy representation of $\psi$ and $h_\gamma$ the  monodromy  of $\psi$ along $\omega$.  
\end{Def}
Since $h$ is a representation it is uniquely determined by its value on the generators of the lattice $\Gamma.$ We denote a choice of generators by $\omega_1$ and $\omega_2.$\\

Following \cite{BLPP} the spectral curve $\Sigma$ of a conformal immersion $f: T^2 \rightarrow S^4$ with trivial normal bundle is  the  
Riemann surface parametrizing all possible monodromies of locally holomorphic sections of $(V/L, D).$ Let $H^0(\widetilde{V/L})$ denote the space of holomorphic sections of $V/L$ with monodromy. 
Consider a section $\psi \in H^0(\widetilde{V/L})$ and let $h$ denote its monodromy representation. For a constant $\lambda \in \H$ the section $\tilde \psi = \psi \lambda$ is also holomorphic. But its monodromy representation is given by $\tilde h = \lambda^{-1} h \lambda$.
Since $\Gamma$ is a lattice, we have 
$$h_{\omega_1} h_{\omega_2} = h_{\omega_2} h_{\omega_1}.$$
Thus the imaginary parts of the quaternions $h_{\omega_1}$ and $h_{\omega_2}$ are linearly dependent over $\R$. 
Therefore it is  always possible to choose a $\lambda$ 
such that $\lambda^{-1} h \lambda : \Gamma \rightarrow \C_*.$ 
\begin{Def}
Let $f$ be a conformally immersed torus in $S^4$ with trivial normal bundle.
 The normalization $\hat \Sigma$ of the analytic variety given by 
$$ Spec(V/L):= \{(a,b) \in \C_* \times \C_* | \text{there exist a } \psi \in H^0(\widetilde{V/L})\text{ with } h_{\omega_1} = a \text{ and } h_{\omega_2} = b\}$$
is called the spectral curve of $f$. 
\end{Def}
Let $ \psi \in H^0(\widetilde{V/L})$ be a holomorphic section with complex monodromy $h.$ Then the multiplication by the quaternion $\jj$ gives a holomorphic section with monodromy $\bar h.$ Thus  it defines a anti-holomorphic involution $\rho$ on $\hat \Sigma.$
If the surface lies in $S^3 \subset S^4$ we get:
\begin{Lem}[\cite{FLPP} or \cite{H} section 8.]\label{sigma11}
Let $f: T^2 \rightarrow S^3$ be a conformal immersion and $V/L$ its quotient bundle. Then  the map
$$\sigma: Spec(V/L) \rightarrow Spec(V/L), h \mapsto h^{-1}$$ is well-defined and holomorphic.
\end{Lem}

The spectral curve $\hat \Sigma$ defined here is non compact. If $\hat \Sigma$ has finite genus, it can be compactified  to a Riemann surface $\Sigma$ by adding two points interchanged by the involution $\rho$, These points are denoted by $0$ and $\infty$.

\begin{Def}
A conformal immersion $f: T^2 \rightarrow S^3$ such that its spectral curve  has finite genus $g$ is called a finite type immersion and $g$ is called its spectral genus.
\end{Def}

Since our interest lies in the class of constrained Willmore tori in the $3-$sphere, whose spectral curve are by \cite{B} connected and have  finite genus, we restrict ourselves to this case in the following.

 \subsection{The Reconstruction}
 The spectral curve itself contains not enough information to reconstruct the immersion. One still needs a $T^2-$family of holomorphic line bundles which comes from the following construction. 
To a generic point $h \in \hat \Sigma$ there exist by \cite{BLPP} a holomorphic section $\psi_h$ in $V/L$ with complex monodromy $h$ unique up to complex scale.
Thus we can assign to  a generic $h \in \hat \Sigma$ the complex line given by $\psi_h \C$. Although $\psi_h$ is not well-defined on the torus, the line given by $\psi_h \C$ is.  The so defined line bundle $\mathcal L\subset \hat \Sigma \times \Gamma(V/L)$ extends holomorphically through the exceptional points of $\hat \Sigma,$ where we have a higher dimensional space of holomorphic sections.  The kernel bundle $\mathcal L$ is compatible with the involutions $\rho$ and $\sigma.$ Since the holomorphic section with monodromy $\bar h$ is given by $\psi_h \jj$ we have $\rho^* \mathcal L = \mathcal L \jj.$ \\

$\mathcal L $ does not extend to $\Sigma$ by the asymptotic properties of $\hat \Sigma$ shown in \cite{BoPP}. This can be repaired by fixing a point $x \in T^2$ and evaluating the  holomorphic section $\psi_h$ at each point of $\Sigma $ in $x. $ In other words, to a fixed $x \in T^2$ we associate to a generic point $h \in \hat \Sigma$ the line given by  $ \psi_h (x) \C.$ The holomorphic line bundle corresponding to this construction is denoted by $\mathcal L_x \subset \Sigma \times \C^2,$ see \cite{BLPP}. 

\begin{The}[\cite{BoPP} theorem 5.6]\label{Jacobi}
The map 
$$\Psi: T^2 \rightarrow Jac(\Sigma), x \mapsto \mathcal L_x \mathcal L_{x_0}^{-1}$$ for a fixed base point $x_0 \in T^2$ is a group homomorphism. 
\end{The}
This theorem means that the conformal immersion $f$ induces a linear map into the Jacobian of its spectral curve. Thus  it is possible to construct $f$ explicitly in terms of algebraic data on $\Sigma.$  

\begin{Rem}
The bundle $\mathcal L$ over $\Sigma$ can be lifted to a line subbundle $\tilde {\mathcal L} $ of $\Sigma \times\Gamma(\widetilde V)$ by assigning to a generic point $h \in \hat \Sigma$ the complex line given by the so called  prolongation $\tilde \psi_h$ of $\psi_h,$ see \cite{BLPP}. The monodromy  of $\tilde \psi_h$ is also $h$.  Further the bundle $\tilde {\mathcal L} _x$, for  a fixed $x \in T^2, $ also extends to $\Sigma.$
\end{Rem}

Let $\Sigma$ be a connected compact Riemann surface with an anti-holomorphic involution $\rho.$ Further fix a real subtorus $Z = \Psi(T^2)$ of dimension $0,1$ or $2$ of the Jacobian of $\Sigma.$ For $x \in T^2$ the line bundle $\tilde {\mathcal L}_x$ over $\Sigma$ is by construction a complex holomorphic line subbundle of $\Sigma \times \C^4.$ Thus it defines a map from $\Sigma$ to $\C P ^3.$ A  quaternionic structure on $\C^4$ is a real linear endomorphism $\jj$ with $\jj^2 = -1$ anti-commuting with $i$.  By fixing such a quaternionic structure $\jj$ on $\Sigma \times \C^4$ we obtain a canonical  isomorphism between $\C^4$ and $\H^2.$ This isomorphism induces a map $\pi_{\H}$ between  $\CP^3$ and $\H P^1$ which is called twistor projection. In our case the quaternionic structure is given by $\rho.$ The main theorem for the reconstruction of a conformally immersed torus is the following.
\begin{The}[\cite{BLPP} theorem 4.2]\label{Immersionreconstruction}
Let $f: T^2 \rightarrow S^3$ be a conformal immersion whose spectral curve $\Sigma$ has finite genus. Then there exist a map
\begin{equation}
 F : T^2 \times \Sigma \rightarrow \C P^3, 
\end{equation}
such that 
\begin{itemize}
\item $F(x, - ): \Sigma \rightarrow \C P^3$ is an algebraic curve, for all $x \in T^2.$\\
\item The original conformal immersion $f : T^2 \rightarrow S^3$ is obtained by the twistor projection of the evaluation of $F$ at the points at infinity:
$$f = \pi_{\H} F(-, 0) = \pi_{\H} F(-, \infty).$$
\\
\end{itemize}
\end{The}

 \section{Reduction to the Equivariant Case}

The spectral curve is the space of possible monodromies of locally holomorphic sections of $V/L$. For equivariant conformal immersions $f: T^2 \rightarrow S^3$ the holomorphic structure $D$ on $V/L$ is translational invariant. Thus the partial differential equation $D\psi = 0,$ where $\psi$ is a section of $V/L$ with monodromy can be reduced to an ordinary differential equation. We want show:

\begin{The}\label{hyperelliptic}
The spectral curve of an equivariant immersion $f: T^2 \rightarrow S^3$ is a double covering of $\C$.
\end{The}

\subsection{The Hopf Field of an Equivariant Immersion}
We consider $S^3 \subset \H$. 
Then we can think of $\H$ lying in $\H P^1$ by the map
$$\H \hookrightarrow \H P^1, x \mapsto [x,1].$$
Recall that we consider a map $f$ into $S^4 \cong \H P^1$ as the quaternionic line bundle given by $L = f^*\mathcal T.$
Thus the section $\varphi:= \pi_{V/L}(1,0)$ of $V/L$ is non vanishing and is thus a trivializing section. The holomorphic structure $D$ is determined by $D \varphi = 0.$ Let $J := \pi_{V/L} S$ denote the complex structure on $V/L$ under which $D$ splits into a $J-$commuting part $\bar \del = \tfrac{1}{2}(D - JDJ)$ and a $J-$anti-commuting part $Q = \tfrac{1}{2}(D + JDJ)$.  While $\bar \del$ defines a complex holomorphic structure on the complex line bundle $E$ given by the $\ii-$eigenspace of $J,$ the Hopf field $Q$ is tensorial. 
Thus  it is uniquely determined by the value of $Q(\varphi).$ 
Since $V/L$ is a quaternionic line bundle and $\varphi$ is a trivializing section, there exists a quaternionic valued function $N$ - the left normal vector of $f$ -  with $N^2 = -1$ with $J \varphi = \varphi N.$ We obtain therefore 
\begin{equation*}
\begin{split}
Q(\varphi) &= \tfrac{1}{2}(D + JDJ)\varphi \\
& = \tfrac{1}{2}J(\varphi d N)''\\
&= \tfrac{1}{4}J(\varphi d N + J \varphi *dN)\\
&=\tfrac{1}{4}\varphi(NdN - *dN),
\end{split}
\end{equation*} where $d$ is the derivative in $\H \cong \R^4.$ 
We want to switch to another trivializing section which behaves nicer with respect to $J$ in order to compute $Q.$ It turns out that under this trivialization of $V/L$ the Hopf field $Q$ is given by the conformal Hopf differential of $f$.
\begin{Lem}
Let $f$ be a conformally immersed torus in $S^3$ and $V/L$ the corresponding quotient bundle. 
Then we can choose a trivializing section with monodromy $\psi \in \Gamma(\widetilde{V/L}) $  with 
\begin{equation*}\label{psi}
\begin{split}
J \psi = \psi \ii \\
\bar \del \psi = 0,
\end{split}
\end{equation*}
where $\bar \del$ is the $J-$commuting part of the holomorphic structure $D$ on $V/L.$ This section $\psi$ is uniquely determined up to multiplication by a complex constant and has monodromy $\pm1$. 
\end{Lem}
\begin{proof}
The $\ii-$eigenspace $E$ of $J$ on $V/L$ is the dual of a spin bundle by \cite{FLPP}. Since $\bar \del$ defines a complex holomorphic structure on $E$, it is  dual to a spin structure and thus there exist a solution to  the equation $\bar \del \psi = 0$ on the  double covering of the torus. Therefore we obtain a section $\psi$ with monodromy $\pm 1$ satisfying the conditions above. This condition fixes $\psi$ up to multiplication with a complex constant. 
\end{proof}
\begin{Rem}
Since $\varphi$ is a  trivializing section of $V/L$, every other section can be written as $\psi = \varphi \lambda,$ where $\lambda$ is quaternionic valued function.  Further we have  $J \varphi= \varphi N$. Then we have
$$J\psi = J (\varphi \lambda) = (J \varphi) \lambda  = \varphi N \lambda =\varphi \lambda \lambda^{-1} N \lambda = \psi \lambda^{-1} N \lambda.$$
Thus $J \psi = \psi \ii$ reduces to 
$$N = \lambda \ii \lambda^{-1}.$$ 
\end{Rem}

\begin{Lem}
Let $f$ be a conformally immersed  equivariant torus in $S^3$ and let $\psi$ be the  trivializing section of $V/L$ given in Lemma \ref{psi}. Then there exist a complex valued function $q$ satisfying
$$Q \psi = \psi d\bar z q \jj ,$$
where $z = x + iy$ is a holomorphic coordinate of $T^2 = \C/\Gamma.$
\end{Lem}
\begin{proof}
The Hopf differential $Q$ is tensorial and satisfies $*Q = - JQ,$ since it takes values in $\bar K (V/L).$ Thus  there is a quaternionic valued function $\tilde q$ with 
$$Q(\psi) = \psi d\bar z \tilde q. $$
Further $Q$ is 
$J-$anti-commuting and 
\begin{equation*}
JQ(\psi) = - Q(J \psi) = - Q(\psi \ii) = - Q (\psi) \ii .
\end{equation*}
This is equivalent to 
\begin{equation*}
\quad J (\psi d \bar z \tilde q) =  \psi (\ii d \bar z \tilde q) = - \psi (d \bar z \tilde q \ii),
\end{equation*}

therefore
\begin{equation*}
\ii \tilde  q = - \tilde q \ii.
\end{equation*}
Since the quaternionic function $\tilde q$  anti-commutes with $\ii$, it takes values in span$\{\jj, \kk\}$ and  there is a complex function $q$ with $\tilde q = q \jj.$  
\end{proof}
\begin{Pro}\label{hpsi}
For an equivariant conformal immersion $f: T^2 \rightarrow S^3$  the function $q$  defined in the previous lemma  is the conformal Hopf differential of the torus.
In particular  $q$ depends only on $y$ and is periodic. 
\end{Pro}

\begin{proof}
The previous Lemma states
$$Q(\psi) = \psi d\bar z q \jj,$$
for a complex valued function $q.$ We want to compute the  function $q$ explicitly.  
Since $Q$ is tensorial and there exist a quaternionic valued function $\lambda $ with $\psi = \varphi \lambda$.   We get 
\begin{equation}\label{Q}
4Q(\psi) =4 Q(\varphi) \lambda = \varphi  (NdN - *dN) \lambda = \psi \lambda^{-1} (NdN - *dN) \lambda.
\end{equation}
\\
We want to make a special choice of $\lambda$ in order to relate $q$ to the conformal Hopf differential of the torus. Then we show that $\psi = \varphi \lambda$ satisfies the conditions of Lemma \ref{psi}. \\

\noindent
Let $f$ be a conformally parametrized $(m,n)-$torus in $S^3 \cong SU(2)$ given by
$$f(x,y) = e^{ \ii l_1x}\gamma(y) e^{ \ii l_2x},$$
see section \ref{Seifert fibrations} and let $(\tilde T, N_{orm}, \tilde B)$ be the Fr\'enet frame of its profile curve, see Lemma \ref{frenet}. We denote the corresponding frame of the torus also by $( \tilde T, N_{orm}, \tilde B).$
By right translation in the Lie group $S^3 \cong SU(2)$ we obtain point wise a orthonormal basis  $(  \tilde T \bar f, N_{orm} \bar f, \tilde B \bar f)$ of $su(2) \cong \text{Im}\H $. We choose $\lambda$ to be the element of $SO(3)$ rotating  this basis to the canonical basis $(\ii, \jj, \kk)$. Let $T :=\tilde T \bar f$ and  $B := \tilde B \bar f.$
We have 
$$\ii =  \lambda^{-1} N_{orm} \bar f \lambda , \quad  \jj =  \lambda^{-1} T\lambda,  \quad \kk =  \lambda^{-1} B\lambda .$$ 

\noindent
Since $f$ is conformal, its left normal vector  $N$ is given by 
 $$*df = Ndf$$
It coincides with the function defined by $J \varphi = \varphi N$ because also $D\pi_{V/L}(0,1) = 0$ and 
$$ \pi_{V/L}(0,1) + \pi_{V/L}(1,0)f= 0_{V/L}.$$

\noindent
The normal vector of the surface is 
$$N_{orm} = \frac{1}{h}\frac{\del f}{\del x} \bar f  \frac{\del f}{\del y}=  -\frac{1}{h} \frac{\del f}{\del x}   \frac{\del \bar f}{\del y} f.$$
The function $N$ is thus the right translation of the normal vector of the surface to $su(2).$ Therefore the choice of $\lambda$ satisfies the condition $J \psi = J \varphi \lambda = \varphi N\lambda = \psi \ii$ and we have $Q (\psi) = \psi d\bar z q \jj,$ for a complex function $q.$\\

Now we compute $\lambda^{-1} (NdN - *dN) \lambda.$
The Fr\'enet equations for the profile curve in $S^3$ gives:  
\begin{equation*}
\begin{split}
N'_{orm} &= - \sqrt{h} \kappa_{S^3}  \tilde T + \sqrt{h}\tau_{S^3} \tilde B\\
\Rightarrow \frac{\del N}{\del y} &= N'_{orm} \bar f - N_{orm} \bar f \frac{\del f}{\del y} \bar f \\
&= - \sqrt{h} \kappa  T - \sqrt{h}\tau B - \sqrt{h} N T \\ &= - \sqrt{h} \kappa  T - \sqrt{h}(\tau+1) B \\
\frac{\del N}{\del x} &= l_1(\ii N - N\ii).
\end{split} 
\end{equation*}
Inserting into equation (\ref{Q}) we get 
\begin{equation}\label{q}
Q \psi 
= l_1 \lambda^{-1}(N \ii N + \ii)\lambda  +  (\sqrt{h} \kappa   +  \sqrt{h}(\tau+1) \ii )\jj dz.
\end{equation}
 The term $N\ii N + \ii$ is perpendicular to $N$ and purely imaginary. And because $(N,  T,  B)$ is a orthonormal basis of Im$\H$ we have
$$ N\ii N + \ii = <N\ii N + \ii, T>T + <N\ii N + \ii, B>B.$$
We compute both summands separately. The multiplication by a unit length quaternion does not change the metric on $\H \cong \R^4.$ Thus
\begin{equation*}
\begin{split}
<N\ii N + \ii, T> &= <N\ii N, NTN> + < N\ii N, T>\\
&=  2<N, \ii B> = -\frac{2l_2}{\sqrt{h}} <N,  \ii f \ii \bar f> .
\end{split}
\end{equation*}
And with $N = N_{orm} \bar f$ we obtain
$$l_1<N\ii N + \ii, T>= - \frac{2l_1 l_2}{\sqrt{h}} <N_{orm}, \ii f \ii> .$$
Now we turn to the second term:
\begin{equation*}
\begin{split}
<N\ii N + \ii, B> &= 2<\ii, B> = - 2l_1 \frac{1}{\sqrt{h}} -2l_2 <\ii ,\frac{f \ii \bar f}{\sqrt{h}} > \\
&= - 2l_1 \frac{1}{\sqrt{h}}  + \frac{2 l_2}{\sqrt{h}} Re(\ii f \ii \bar f) .
\end{split}
\end{equation*}
Note that for equivariant tori $Re( \ii f  \ii \bar f) = Re( \ii \gamma  \ii \bar \gamma)$. Moreover let $\gamma = \gamma_1 + \jj \gamma_2$ with complex functions $\gamma_1$ and $\gamma_2.$ Then we have $h = m^2|\gamma_1|^2 + n^2|\gamma_2|^2$ and of course $|\gamma_1|^2 + |\gamma_2|^2 = 1.$ Thus we get
\begin{equation*}
\begin{split}
<N\ii N + \ii, B> &= - \frac{2 }{\sqrt{h}}(m|\gamma_1|^2 - n|\gamma_2|^2)\\
\Rightarrow l_1 <N\ii N + \ii, B>  &= - \frac{1 }{h}((m-n)m|\gamma_1|^2 - (m-n)n|\gamma_2|^2) \\
&= \tfrac{mn}{\sqrt{h}} -  \sqrt{h}.
\end{split}
\end{equation*}

Put these results into equation (\ref{q})   we obtain
$$q =  \frac{1}{4}\left(\sqrt{h} \kappa_{S^3} - \frac{2l_1l_2}{\sqrt{h}} <N_{orm},  \ii\gamma  \ii>   + \ii\frac{2mn}{\sqrt{h}} \right) = \frac{1}{4}\left(\kappa_{m,n} + \ii \frac{2mn}{\sqrt{h}}\right),$$
which is the conformal Hopf differential of $f$.

It remains to show that we can adjust the section $\psi $ in order to get $\bar \del \tilde \psi = 0$ preserving the property $J \tilde \psi = \tilde \psi \ii$ and the function $q.$ We can still multiply by a real valued function. 
The condition $\bar \del \tilde  \psi = 0$ holds if and only if  $D (\tilde \psi)$ anti-commutes with $\ii.$ We have
$$D(\psi) = D (\varphi \lambda) = (\varphi d \lambda)'' =\ \psi \lambda^{-1} d \lambda  +  J \psi \lambda^{-1} *d \lambda =  \psi ( \lambda^{-1} d \lambda  +    \ii \lambda^{-1} *d \lambda). $$
Because of the frame equations  of the profile curve we have
$$ B' (0,y)  =  (\sqrt{h} \tau + \sqrt{h})N(0,y).$$ 
Substituting this equation into the derivative of the equation $\lambda^{-1} B \lambda = \kk,$ which is valid by the definition of $\lambda, $ gives
$$\lambda^{-1}(y) \lambda'(y)  = \jj v(y),$$ with a complex valued function $v.$
Since $\lambda(x,y) = e^{ \ii l_1 x} \lambda(y) ,$ we get
$$\lambda^{-1} d \lambda  +    \ii \lambda^{-1} *d \lambda = d\bar z (l_1 \lambda^{-1}  \ii \lambda + \jj v(y)). $$
Thus for $r = e^{ \int_0^y  l_1<\lambda^{-1}  \ii \lambda, \ii> ds}$ and $\tilde \lambda  = \lambda r$ the section $\tilde \psi =  \varphi \tilde \lambda$ is holomorphic with respect to $\bar \del.$ 
\end{proof}

\subsection{The Spectral Parameter}
The next step towards the spectral curve of an equivariant torus is to exploit the translational invariance of the holomorphic structure $D = \bar\del + Q.$
Let $f: T^2 = \C /\Gamma \rightarrow S^3$ be a conformal immersion. Choose $\omega_1 \in \R$ and $\omega_2$ be generators of the lattice $\Gamma$ and let $V/L$ be the quotient bundle associated to $f.$ Recall that the spectral curve is defined to be the normalization of the analytic variety given by 
$$ \{ (h_1 , h_2) \subset \C_* \times \C_* | \exists \varphi \in H^0(\widetilde{V/L}) \text{ with monodromy  } h_{\omega_1} = h_1 , h_{\omega_2} = h_2 \}.$$
In order to compute the spectral curve, it is thus sufficient to know the generic points. \\

Let $\psi$ be the trivializing section of $V/L$ with monodromy of Lemma \ref{psi} and Proposition \ref{hpsi}. To every holomorphic section $\tilde \psi$ of $V/L$ with monodromy there is a quaternionic valued function $u $ with $\tilde \psi = \psi u$ and the equation $D (\psi u) = 0$ reduces to
\begin{equation}\label{u}
  \bar \del u  + (d\bar z  q \jj) u  = 0,
\end{equation}
where $\bar \del$ is the ordinary holomorphic structure for functions.  By identifying $\H$ with $\C \oplus  \jj \C $ we get for $u = u_1 +  \jj u_2 ,$ the matrix notation of equation (\ref{u}).

$$D u := \dvector{\bar \del & \bar  q\\   - q &\del} \dvector{u_1 \\  u_2}  = 0.$$

Since the potential $q$ depends only on $y$, the differential operator $D$ is translational invariant. Thus if $u(x,y) $ satisfies $D u = 0,$
then  also $D u(x+x_0, y) = 0,$ for an arbitrary constant $x_0 \in \R$. Obviously both solutions have the same monodromy. By   \cite{BLPP} 
the space of holomorphic sections with 
monodromies $(h_{\omega_1} , h_{\omega_2})$  is generically complex $1$ dimensional. Thus, at a generic point of the spectral curve, we have that for all $x_0 \in \R$
 $$u(x+x_0, y) =  u(x,y) A_{x_0},$$
 where $A_{x_0}$ is a complex constant.   The so defined function
 $A: \R \rightarrow \C$ is given by $A_x = e^{a x}$ for $a \in \C$ with  $e^{a \omega_1} = h_1.$

Thus we can  use the ansatz $u_i = e^{a x} \tilde u_i(y),$ where $a \in \C$ is a constant representing the monodromy in $x-$direction. This yields
\begin{equation}\label{dirac}
Du = \left( \frac{\del}{\del y} + \dvector{ - i a &  2 i \bar q\\   2 i q &  ia }\right) \dvector{\tilde u_2 \\  \tilde u_1} = 0. 
\end{equation}
\noindent
We have shown the following Lemma.
\begin{Lem}
Let $f: T^2 \rightarrow S^3$ be an equivariant and conformal immersion and $V/L$ the associated quotient bundle. The spectral curve of $f$ is determined by the kernel of the  family of ordinary linear differential operators 
$$D_a:=  \frac{\del}{\del y} + \dvector{ - i a &   2 i \bar q\\  2 i q &  i a }, \ \ a \in \C.$$
\end{Lem}

\begin{Lem}
For generic $a \in \C$
the matrix $H(a)$ is diagonalizable and has distinct eigenvalues.
\end{Lem}
\begin{proof}
If $H(a)$ is not diagonalizable,  then $H(a)$ still is trigonalizable and has  only one eigenvalue $1$  or $-1$ with algebraic multiplicity $2.$  
Since $Spec(V/L)$ is an analytic variety the map 
$$p_1: Spec(V/L) \rightarrow \C_*,  (b,a) \mapsto b$$
is holomorphic. Therefore, if  there exist a open set $U\subset \C $ such that $H(a)$ is not diagonalizable for $a \in U$ then the map $p_1$ must be constant on $Spec(V/L)$. Thus all solutions to $D_a$ have either monodromy $1$ or $-1.$ Now let $a \in i \R$ and $u_a(y)$ be an eigensolution of $D_a.$ Further, let $\psi$ be the trivializing section with monodromy introduced in Proposition \ref{hpsi}. Then $\psi e^{ax} u_a(y)$ is  a holomorphic section with monodromy $(e^{a\omega_1}, 1)$ or $(e^{a\omega_1}, -1)$   in $V/L.$  In either case for $a \in \tfrac{\omega_1 i }{2\pi}\Z$ we obtain infinitely many holomorphic sections with monodromy $(1, -1)$ or $(1,1)$ which is not possible by section 3.1 of  \cite{BLPP}.
\end{proof}

\noindent
Now we can proof Theorem \ref{hyperelliptic}.
\begin{proof}
The spectral curve is determined by all possible monodromies of solutions to equation (\ref{dirac}).  Since (\ref{dirac}) is an ordinary linear differential equation, we have for arbitrary $a \in \C$ two linear independent solutions.  
 Let $\Phi(a)$  be the fundamental solution matrix to (\ref{dirac}). 
Then we have
$$\Phi(y + \omega) = H(a) \Phi(y),$$ where $\omega \in \R$ is the period of the potential $q,$ and  $H(a)$ is a $SL(2, \C)$ matrix independent of $y.$ 
The solutions with monodromy of (\ref{dirac}) are exactly the eigensolutions of $H(a)$.
 Therefore the spectral curve is the normalization of the variety
$$Spec(V/L) = \{(b, a) | a \in \C \text{ and }  b \text{ eigenvalue of } H(a)\}.$$

The normalization of $Spec(V/L)$ is a double covering of $\C$, because to a generic spectral parameter $a \in \C$ we have  two distinct eigenvalues $H(a)$ and thus two different points in $Spec(V/L).$ 
If the spectral curve $\hat \Sigma$ is of finite genus then it can be compactified  to $\Sigma$ by adding two points at infinity. In this case the ends of the spectral curve corresponds to the points over $a = \infty$ and we obtain that $\Sigma$ is a hyperelliptic curve, which is not branched over $a = \infty$. 
\end{proof}

\begin{Rem}
If the spectral curve has finite genus, then it is given by the normalization of the equation $$\eta^2 = Tr(H(a))^2 -4 =: P(a).$$
By definition of the spectral curve for immersions into $S^3$ we have two involutions $\rho$ and $\sigma$ on the spectral curve.  These involutions induce involutions on the spectral parameter plane which are given by $\rho : a \mapsto \bar a$ and $\sigma : a \mapsto -a.$
Thus the polynomial $P$ is even and has real coefficients.
\end{Rem}

\subsection{Polynomial Killing Fields}
The definition of a spectral curve for the family of Dirac operators $D_a$ is the same as  for the focussing nonlinear Schr\"odinger  equation. In the case of finite spectral genus, hyperelliptic solutions of the nonlinear Schr\"odinger equation are constructed in \cite{Pr}, under some further restrictions.  We use another approach here.
For equivariant tori  we show the existence of a endomorphisms-valued $1-$form $X$ depending on the spectral parameter $a$ and the profile curve parameter $y$, whose eigenlines for generic $a \in \C$ coincides with the space of eigensolutions of $D_a.$ If $X$ extends meromorphically to $a= \infty,$ i.e., if $X$ is polynomial in $a$, we call $X$ a polynomial Killing field, since it preserves the spectral curve in $y$ direction (see lemma \ref{DetXpreserved}).  The equation on the polynomial Killing field is linear and explicitly solvable and gives rise to a hierarchy of ordinary differential equations on the potential $q.$ 
In order to define a polynomial Killing field we need the following proposition from  \cite{BoPP} (Proposition 3.1).
\begin{Pro}[\cite{BoPP}]
For a family of  elliptic operators, which depends holomorphically on a parameter in a connected $1-$dimensional complex manifold $M,$ the minimal kernel dimension is generic and attained away from isolated points $p_i \in N \subset M.$ Further the vector bundle over $M\setminus N$ defined by the kernels of the elliptic operators extends through the isolated points with  higher dimensional kernel and is holomorphic.
\end{Pro}

\begin{Cor}
Let $f : T^2 \rightarrow S^3$ be an equivariant conformal immersion and $D_a$  the associated family of Dirac type operators on $W := S^1 \times \C^2$. For a fixed $a \in \C$ we define
$$\mathcal E_a := \{X \in \Gamma(\text{End}_0(W)) \ | \ D_a X = 0\},$$
where $\text{End}_0(W)$ denote the trace free endomorphisms of $W$ and $D_a$ is the induced differential on $\Gamma(\text{End}_0(W))$ by Leibniz rule. 
Then there is a holomorphic line bundle $\mathcal E$ over $\C$ whose fiber over a generic spectral parameter $a \in \C$ coincides with $\mathcal E_a.$
\end{Cor}
\begin{proof}
We want to show that the bundle $\mathcal E$ is of rank $1.$
 Let $L = \dvector{ - i a &  2 i \bar q\\  2 i q &  i a }$. 
The equation $D_a X = 0$ is equivalent to 
\begin{equation}\label{gleichung1}
X'  =  [X, L],
\end{equation}
where the derivative is taken with respect to the profile curve parameter $y.$ 
This is a first order ordinary differential equation and therefore $X$ is fully determined by its initial value at $y = 0.$ We denote the initial value of $X$ by $X_0.$\\

At a generic point $a \in \C$ the holonomy of $D_a$ on $W$ is diagonalizable and has distinct eigenvalues $\mu^{\pm1},$ $\mu \neq \pm 1$. Let $L^+_a$ and $L^-_a$ define the corresponding eigenlines. With respect to the splitting $W = L^+_a \oplus L^-_a$ we get that  $X(a)= \dvector{- s & v \\ w & s},$ where $s \in \Gamma(\C)$, $v \in \Gamma(L^-_a \otimes (L^+_a)^*)$  and $w \in \Gamma(L^+_a \otimes (L^-_a)^*).$ Because of equation (\ref{gleichung1}) the sections $s, v, w$ are parallel (and periodic) with respect to the induced connection on the corresponding line bundles. This implies that $s$ is constant and $v=w=0$.
Therefore $X$ is fixed by the initial value of $s$ and  $\mathcal E$ is a complex line bundle over $\C$. 
\end{proof}

\noindent
We want now to restrict to the case where the line bundle $\mathcal E$ extends through $a = \infty$ to a holomorphic line bundle over $\C P^1.$ Then $\mathcal E$ has finite negative degree as a holomorphic sub bundle of the trivial endomorphisms bundle which can be computed by
$$\text{deg }  \mathcal E =   \sum_{p\in \C P^1}\text{ord}_pY,$$
for any meromorphic section $Y$  of $\mathcal E.$  A holomorphic line bundle over $\C P^1$ is determined by its degree and is holomorphic isomorphic to a point bundle $L ((\deg \mathcal E )\infty)$.  
Thus there is a section $X$ which is holomorphic and non vanishing for $a \in \C.$ 
\begin{Def}
A non vanishing meromorphic section $X$ of $\mathcal E$ whose only pole is at $a = \infty $ is called a polynomial Killing field of $\Sigma$.
\end{Def}
\noindent
By definition $X$ is polynomial in $a \in \C P^1$ and thus it is given by
$$X = \sum_{i = 0}^{-\deg \mathcal E} X_i a^i ,$$
where $X_i$ are elements of  $\Gamma($End$_0(W))$ depending on $y$ only. 
For $a \in \R$ we have that the holonomy of $D_a$ is $SU(2)-$valued and the corresponding eigenlines are perpendicular. With respect to a  basis of eigenvectors we can choose the section $s_0 \in L^1_a|_{y=0}$ to be purely imaginary-valued for $a \in  \R.$ Then $X_0$ is $su(W_0)$  valued. Since this property is preserved by the equation  $X' =  [X, L]$,  the polynomial Killing field $X $ is also $su(W)$ valued.

\begin{Lem}\label{DetXpreserved}
Let $X$ be a polynomial Killing field of $\Sigma.$
Then the equation $X' =  [X, L]$ 
 preserves the polynomial $\text{det}(X).$
\end{Lem}
\begin{proof}
Since tr$(X) = 0$ we get $2$det$(X)$ = tr$X^2.$ Differentiating both sides yields
$$(\text{det}X)' = \text{tr}(X') X = \text{tr}([X, L] X) = \text{tr}(X L X - L X^2) = 0. $$
\end{proof}
The polynomial Killing field $X$  has degree $p+1.$ Therefore we get that det$X$ is a polynomial of degree $2p + 2$. The equation 
$$\eta^2 = \text{det } X$$
thus defines  a possibly singular algebraic curve of genus $p.$
\begin{The}\label{deg}
The normalization of the algebraic curve given by $$\eta^2 = \text{ det } X$$ is the spectral curve.
\end{The}
\begin{proof}
The eigenlines of the polynomial Killing field for generic $a \in \C$ gives exactly the solutions with monodromy  to the equation $D_a u = 0.$ The branch points of the spectral curve are given by those points where these eigenlines  coalesce to an odd order. At these points $a \in \C$ the polynomial Killing field is not diagonalizable and therefore it has only one eigenvalue, which must be $0$. Since tr$X(a) = 0$ and $X$ is non vanishing , we get that $X(a) $ is conjugate to the matrix $\dvector{0 & 1\\ 0 & 0}.$ Thus det$X$  has an odd order zero at $a \in \C P^1$ if and only if the spectral curve is branched over $a$.  
\end{proof}

\begin{Def}
Let X be the polynomial Killing field of an immersion. Then the number $p :=  - \deg \mathcal E - 1 $ is called the arithmetic spectral genus of the immersion.
\end{Def}

\begin{Lem}\label{X}
The coefficient $X_{p+1}$ and $X_p$ of a polynomial Killing field  $X$  can be chosen  to be
$$X_{p+1} = \dvector{- i & 0\\ 0 & i } \quad X_p =  \dvector{- i b  & 2 i \bar q\\ 2 i q & i b}.$$
\end{Lem}
\begin{proof}
The degree of the polynomial Killing field is constant in $y$. The Lax type equation 
$$X' = [X, L]$$
yields differential equation on each coefficient $X_i$ of $X$. Since $L = \dvector{-i & 0\\ 0 & i } a  + \dvector{0  &  2 i \bar q\\  2 i q & 0}$  we obtain:
$$X_i' = [X_i, \dvector{0  &   2 i \bar q\\   2 i q & 0}] + [X_{i-1}, \dvector{ -i & 0\\ 0 & i }], \quad \text{ for } i = 1, ..., p+1,$$
Further since $X_{p+2} = 0$ we get that 
$$[X_{p+1}, \dvector{ -i & 0\\ 0 & i  }] = 0.$$
Thus  with out loss of generality $X_{p+1}   = \dvector{ -s & 0\\ 0 &  s }.$ The term  $-s^2$ is the top coefficient of the polynomial det$X$, which is constant along $y.$ Therefore $s$ is a constant and we can normalize it to be $i.$
Then we obtain
$$X_{p+1}' = [X_{p+1}, \dvector{0  &  2 i \bar q\\   2 i q & 0}] + [X_{p}, \dvector{ -i & 0\\ 0 & i }] = 0,$$
and this yields $X_{p} =  \dvector{ -i b  &  2 i \bar q\\  2 i q & i b}.$
\end{proof} 

Now  we  build in the symmetry of the spectral curve coming from an immersion $f: T^2 \rightarrow S^3.$ It is given by the involution $\sigma$, see Lemma \ref{sigma11}. We obtain
$$\text{det}(X(a)) = \text{det} (X(-a)).$$
The branch points are then symmetric with respect to the real axis and the polynomial det $X$ is even.  
Since $$X = \dvector{ -i & 0\\ 0 &  i  }a^{p+1} + \dvector{-i b  &   2 i \bar q\\  2 i q & i b}a^{p}  + \text{ lower oder terms,}$$
the determinant is given by $$\text{det}X =  a^{2p+2} + b a^{2p+1} + \text{ lower order terms. }$$
Therefore $b = 0.$
 The involution $\rho$ gives that the branch points are symmetric with respect to the imaginary axis, which is satisfied for  $X_i \in \Gamma(su(W))$.

\subsection{Equivariant Constrained Willmore Immersion}
Comparing the differential equations on $q$ given by $X' = [X,L]$ to the Euler-Lagrange equations of constrained Willmore tori we want to show the following theorem, which summarizes the Propositions \ref{CMC2}, \ref{Hopf2} and \ref{rest}.
\begin{The}
The spectral genus of equivariant constrained Willmore tori in $S^3$  is at most $3.$ In particular,
\begin{itemize}
\item equivariant tori in $S^3$ have spectral genus $g \leq1$ if and only if they are isothermic and constrained Willmore and thus CMC  in a space form. 
\item Non isothermic equivariant tori in $S^3$ have (arithmetic) spectral genus $2$ if and only if they are constrained Willmore and associated to a Hopf cylinder as constrained Willmore surfaces. 
\item Non isothermic equivariant tori in $S^3$ with (arithmetic) spectral genus $3$ are constrained Willmore.
\end{itemize}
\end{The}

\begin{Lem}
An equivariant torus $f$ in $S^3$ has spectral genus $0$ if and only if $q \equiv const \neq 0$, i.e.,  if $f$ is homogenous.
\end{Lem}
\begin{proof}
The polynomial Killing field of an equivariant torus in $S^3$ with spectral genus $0$ is given by 
$X = L.$
 Then  we get  
$X' = [X, L] = 0$ and $X$ is constant. Therefore 
 \begin{equation}\label{flow1}
 q = const.
 \end{equation} This Killing field has no zeros, if and only if $q \neq 0$.
\end{proof}

\begin{Pro}\label{CMC2}
An equivariant  torus in $S^3$ has spectral genus $1$ if and only if it is CMC in a space form and not homogenous. 
\end{Pro}
\begin{proof} 
Let $f$ be an equivariant conformal immersion. The genus of the spectral curve is $1$ if and only if there is a polynomial Killing field of degree $2$ satisfying
$X' = [X, L].$
By Lemma \ref{X} we have that such a polynomial Killing field is given by  
$$X =  \dvector{-i & 0\\ 0& i } a^2+   \dvector{0 &   2i \bar q\\  2i q & 0 } a +  \dvector{-i b_0 & 2i \bar p_0\\  2i p_0 & i b_0 }.$$
The equation $X' = [X, L]$ gives:
\begin{equation*}
\begin{split}
X_0'  &= \dvector{i b_0' &  - 2i \bar p_0'\\  - 2i p_0' & -i b_0' }= [X_0, X_{1}] = \dvector{ 4 p_0 \bar q - 4 \bar p_0 q  & 4  b_0 \bar q\\   - 4  b_0 q & -4 p_0 \bar q + 4 \bar p_0 q }\\
X_1' &=  \dvector{0 & - 2 i \bar q'\\ - 2 i q' & 0 } = [X_0, X_2] = \dvector{ 0  &  - 4\bar p_0\\  4 p_0 & 0 }\\
X_2' &= 0.
\end{split}
\end{equation*}
Thus we obtain 
\begin{equation}
\begin{split}
2p_0 &= - i q',\\
p_0' &= - 2 i b_0q,\\
i b_0'  &= 4 p_0 \bar q - 4 \bar p_0 q,
\end{split}
\end{equation}
Therefore
\begin{equation}
 b_0 = - 2  |q|^2 - 2 c, \text{ for some real constant }c.
\end{equation}
This gives
\begin{equation}\label{flow2}
 q'' + 8(|q|^2  + c)q = 0,
\end{equation}
The constraint that det$X$ is an even polynomial yields
\begin{equation}\label{symmetry1}
q' \bar q - \bar q' q = 0.
\end{equation}
Equation  (\ref{symmetry1}) is satisfied if and only if there exists a $\mu \in S^1$ such that  $q\mu$ is real valued. Hence the function $\xi$ of equation \ref{EL} is constant. Then
equation (\ref{flow2}) is equivalent to equation (\ref{EL}) and  there exist to every solution $q$ of  (\ref{flow2}) a surface $f$ with conformal Hopf differential $q$. This surface
 is isothermic and constrained Willmore and thus CMC in a space form.\\

On the other hand, if $f$ is isothermic and constrained Willmore, its conformal Hopf differential satisfies the equation (\ref{symmetry1}) and solves
$$q'' + 8(|q|^2  + c)q = 0.$$
To such a solution $q$
 we can define a polynomial Killing field
$$X = \dvector{-i & 0\\ 0&  i } a^2+   \dvector{0 & 2 i \bar q\\ 2 i q & 0 } a +  \dvector{ -i b_0 & 2 i \bar p_0\\ 2 i p_0 & - i b_0 }$$
with $p_0 = - \tfrac{i}{2} q'$ and $b_0=  - 2  |q|^2 - 2 c$.
By construction $X$ satisfies the equation $X' = [X, L]$. If the so defined polynomial Killing field $X$ would have a zero for $a \in \C$ then there exist another section of $\mathcal E$ with degree $1$ without any zeros for all $a \in \C$. Then $q$ is constant and the corresponding $f$ homogenous by the previous lemma.
\end{proof}

\begin{Pro}\label{Hopf2}
An equivariant and non-isothermic torus in $S^3$ has (arithmetic) spectral genus $2$  if and only if it lies in the associated family of a constrained Willmore Hopf cylinder.
\end{Pro}
\begin{Rem}
We call a surface in $S^3$ a Hopf cylinder if it is  the preimage of a (not necessarily closed) curve on $S^2$ under the Hopf fibration. 
\end{Rem}
\begin{proof}
For spectral genus $2$ solutions the  polynomial Killing field is given by
 $$X = \dvector{- i & 0\\ 0& i } a^3+   \dvector{0 &  2 i \bar q\\  2 i q & 0 } a^2 +  \dvector{- i b_0 &  2 i \bar p_0\\ 2 i p_0 &  i b_0 }a +  \dvector{- i b_1 &  2 i \bar p_1\\ 2 i p_1 &  i b_1 }$$ 
 The equation $X' = [X, L]$ gives by a straight forward calculation 
 \begin{equation}\label{constants}
 \begin{split}
 2p_0 &= - i q',\\
i b_0'  &= 4 p_0 \bar q - 4 \bar p_0 q,\\
2 p_1 &=  2  b_0 q  -  i p_0', \\
 i b_1 '  &=  -  2 i p_0'  \bar q  - 2 i \bar p_0' q, \\
 \end{split}
 \end{equation}
 Therefore we get
 \begin{equation*}
 \begin{split}
 b_0 &= - 2  |q|^2 - 2 c, \text{ for some real constant }c. \\
\text{and } \quad b_1 &=  \bar q' q - q'\bar q + d , \text{ for some real constant }d. 
\end{split}
\end{equation*}
Thus
\begin{equation}\label{flow3}
q''' + 24 |q|^2q' + 8 c q' + d = 0.
\end{equation}
 This equation is known to be the stationary mKdV equation. 
 Further the $\sigma-$symmetry gives:
 \begin{equation}\label{sigma1}
  b_1 +  2p_0 \bar q + 2\bar p_0 q = 0 \text{ and } b_1b_0 + 2 p_0 \bar p_1 + 2\bar p_0 p_1  = 0.
  \end{equation}
The first equation is equivalent to $d = 0.$ And  the second equation gives:
$$p_0'\bar p_0 - p_0 \bar p_0' = 0 ,$$
which holds if and only if there is a $ \mu \in S^1$ with $\mu p_0 \in \ii \R.$  Since $2p_0 = i q',$  
we obtain for the corresponding $q$ that there exist a $\mu$ with  $q\mu  = \kappa + ir $ for a real valued function $\kappa$ and a real constant $r \geq 0$. If $r= 0,$ then the corresponding surface, which always exists,  is isothermic but not constrained Willmore. For $r \neq 0$  the surface with conformal Hopf differential $q = \kappa +  ir$ solving equation (\ref{flow3}) with $d = 0$ satisfies (\ref{EL}) with $\xi = i r\kappa.$ Thus it is a constrained Willmore Hopf cylinder.\\

For a non isothermic constrained Willmore immersion $f:\C \rightarrow S^3$ which lies in the associated family of a constrained Willmore Hopf cylinder its conformal Hopf differential $q$ solves
$$q''' + 24 |q|^2q' + 8 c q'  = 0.$$
We can define a polynomial Killing field
$$X = \dvector{i & 0\\ 0& -i } a^3+   \dvector{0 & - 2 i \bar q\\ - 2 i q & 0 } a^2 +  \dvector{ i b_0 & - 2 i \bar p_0\\ - 2 i p_0 & - i b_0 }a +  \dvector{ i b_1 & - 2 i \bar p_1\\ 2 i p_1 & - i b_1}$$ 
such that the entries satisfy
the equations in (\ref{constants}). 
Then $X' = [X, L]$. Since the solutions of spectral genus $0$ and $1$ are isothermic $q$ has spectral genus $2.$ 
\end{proof}
\begin{Rem}
In contrast to the spectral genus $1$ case the constrained Willmore Hopf tori can have closed solutions with singular spectral curve obtained by simple factor dressing of a multi covered circle. Nevertheless, the arithmetic genus $p$ of the constrained Willmore Hopf tori satisfies $p \leq 2,$ since the profile curve is always constrained elastic. 
\end{Rem}
\begin{Pro}\label{rest}
Let $f: T^2 \rightarrow S^3$ be an equivariant constrained Willmore torus, then its (arithmetic) spectral genus is at most $3.$ Moreover, every non isothermic equivariant torus of spectral genus $3$ is constrained Willmore.
\end{Pro}
\begin{proof}
 An equivariant torus has arithmetic spectral genus $p = 3$ if and only if it has a polynomial Killing field of the form:
  \begin{equation*}
  \begin{split}
  X &= \dvector{- i & 0\\ 0& i } a^4+   \dvector{0 & 2 i \bar q\\ 2 i q & 0 } a^3 +  \dvector{ -i b_0 & 2 i \bar p_0\\ 2 i p_0 &  i b_0 }a^2\\ &+\dvector{ -i b_1 &  2 i \bar p_1\\  2 i p_1 &  i b_1 }a +   \dvector{- i b_2 & 2 i \bar p_2\\ 2 i p_2 &  i b_2 }.
  \end{split}
  \end{equation*}
 \noindent
 Again the equation $X' = [X,L]$ yields equations for each entry of the $X_i.$
We obtain
\begin{equation}\label{constants2}
\begin{split}
 2p_0 &= - i q'\\
i b_0'  &= 4 p_0 \bar q - 4 \bar p_0 q,\\
\Rightarrow b_0 &= - 2  |q|^2 - 2 c, \text{ for some real constant }c. \\
2 p_1 &=  2  b_0 q  -  i p_0' \\
 i b_1 '  &=  -  2 i p_0'  \bar q  - 2 i \bar p_0' q, \\
 \Rightarrow b_1 &=  \bar q' q - q'\bar q + d , \text{ for some real constant }d.\\
 2 p_2 &= 2  b_1 q  -  i p_1'\\
 i b_2' &=  -  2 i p_1'  \bar q  - 2 i \bar p_1' q,\\
 \Rightarrow b_2 &= 6 |q|^4 + 2 c |q|^2   + \tfrac{1}{2}(q'' \bar q + \bar q'' q - q' \bar q')+ e , \text{ for some real constant }e.
\end{split}
\end{equation}
By $\sigma-$symmetry we have
\begin{equation}\label{g3symmetry}
\begin{split}
&b_1 + 2 p_0\bar q + 2 \bar p_0q =0  \quad \quad
\Leftrightarrow \quad \quad
d = 0,\\
 &b_1b_0 + 2 p_2\bar q + 2 \bar p_2 q  +2 p_0\bar p_1 + 2 \bar p_0 p_1=0, \\
 &b_2b_1 + 2 p_1\bar p_2 + 2 \bar p_1 p_2 =0 . 
\end{split}
\end{equation}
Thus we obtain the following differential equation for  $q$:
 \begin{equation}\label{4th}
q'''' + 96 |q|^4q + 16 \bar q' q'q + 24 (q')^2 \bar q + 8\bar q'' q^2  + 32 |q|^2q'' + 8c(q''+ 8 |q|^2q)+ 16e q =0, 
 \end{equation}
 with real constants $c$ and $e.$\\
 
The second condition of (\ref{g3symmetry})  is equivalent to 
 \begin{equation}\label{symmetry}
 \Im\left(q''' \bar q + 24 |q|^2 \bar q q' + 8 c \bar q q' + \bar q'' q' \right)= 0.
 \end{equation}
Note that this condition holds if $q$ is real. Further if there is a $\mu \in S^1$
with $q' \mu$ is real then the condition above reduces to equation (\ref{flow3}). Thus there is no spectral genus $3$ solution with $q = \kappa + ir$ such that det$X$ is even and real. \\

We want to show that all solutions $q$ to the Euler-Lagrange equation of an equivariant constrained Willmore torus satisfies the equation (\ref{4th})  for certain parameters depending on the parameters of the Euler-Lagrange equation (for every compatible $\xi$). For potentials $q$ coming from a surface in $S^3$, the corresponding spectral curve has by definition the $\sigma-$symmetry.  Thus $q$ satisfies (\ref{g3symmetry}). Let $q$ be any solution to \eqref{EL}, then also $q \mu$ satisfies \eqref{EL} with parameters $C_\mu, \lambda_\mu$ and $\xi_\mu$ (see section 2.3). Thus there  exist a $\mu$ such that  $\lambda_\mu$ is real. We restrict to this case $\lambda \in \R$ in the following, since the space of solutions to equation \eqref{4th} is also invariant under the multiplication by $\mu \in S^1.$\\

\noindent
The Euler-Lagrange equation for equivariant constrained Willmore tori in $S^3$ stated in (\ref{EL}) has order $2,$ thus it is necessary to differentiate twice to obtain:
 \begin{equation}\label{4}
 q'''' + 24 \bar q' q'q + 24 (q')^2 \bar q + 4\bar q'' q^2  + 20 |q|^2q'' + 8Cq'' - 8 \xi q'' - 2\lambda \Re (  q'' ) = 0.
 \end{equation}

\noindent
Subtracting this equation from equation \eqref{4th} we obtain
$$96 |q|^4q - 8 \bar q' q' q + 4 \bar q'' q^2 + 12 |q|^2 q'' + 8 c (q'' + 8 |q|^2q) + 16 eq - 8Cq''+ 8 \xi q'' + 2  \lambda \Re(q'') =^{!} 0 .$$
Now we can use the Euler-Lagrange equation \eqref{EL} again to get rid of the second derivatives of $q$. Then
\begin{equation}\label{comp}
\begin{split}
&- 8 \bar q' q' q - 32 |q|^4q + 64 \xi^2 q- 64 C |q|^2q + 8 q^2 \lambda\Re(q) + 8 |q|^2\lambda \Re(q) \\
&+ (16e + 64C^2 + 64cC)q + (16c - 32C +4 \lambda)(4 \xi q + \lambda\Re( q)) =^{!} 0.
\end{split}
\end{equation}
A further computation shows that one can integrate the real and imaginary part of equation (\ref{EL}) once. This yields the equation

 $$\bar q' q' = - 4 |q|^4 + 8 \xi^2 - 8 C |q|^2 + 2\lambda \Re(q)^2  - \tilde d.$$
 with for a real constant $\tilde d.$ \\

 \noindent
 Using this equality \eqref{comp} becomes
 
 \begin{equation}
(16e + 64C^2 + 64cC + 8 \tilde d)q + (16c - 32C +4 \lambda)(4 \xi q + \lambda\Re( q)) =^{!} 0,
 \end{equation}
 \noindent
 which holds if and only if we choose the parameters $c$ and $e$ such that
 \begin{equation}
 \begin{split}
& 2 e + \tilde d + 8C^2 + 8cC = 0.\\
&4 c - 8C + \lambda = 0.
 \end{split}
 \end{equation}
\end{proof}

Every solution of (\ref{4th}) is fully determined by the initial value of $q,$ $q'$, $q''$ and $q'''.$ 
 By condition (\ref{g3symmetry})  these initial values also determines the constants $c$ and $e$. Thus the space of solutions $q$ of spectral genus $g \leq 3$ such that the spectral curve is defined by an even and real polynomial is $8$ dimensional.  Since every equivariant constrained Willmore solution has spectral genus $g \leq 3,$ a solution $q$ of (\ref{4th}) solves the Euler-Lagrange equation (\ref{EL}) if and only if there exist parameters ($\lambda$ and $C$ and a initial value for $\xi$) such that its initial values  $q_0,$ $q_0'$, $q_0''$ and $q_0'''$ satisfies (\ref{EL}).
This is always the case for non isothermic solutions by the following argument. If $q \mu$ is never real valued for $\mu \in S^1,$  there there exists always a point $y_0$ such that  the initial values of the solution $q = q_1 + i q_2$ satisfies  $q_1(y_0) \neq 0$, $q_2(y_0)\neq 0$ and $q_1' (y_0) \neq 0$, $q_2'(y_0)  \neq 0$ and $\frac {q_1(y_0)}{q_2(y_0)} \neq \frac {q_1'(y_0)}{q_2'(y_0)}.$ Thus at $y_0$ we get all possible initial values by choosing appropriate constants $\lambda,$ $C$ and $\xi_0$. Therefore all non isothermic solutions of (\ref{4}) are constrained Willmore.
\bibliographystyle{amsplain}

\end{document}